\documentclass[12pt]{amsart}

\usepackage{amsthm,amssymb,verbatim,color}
\usepackage{rotating}

\numberwithin{equation}{section}
\newtheorem{thm}{Theorem}[section]

\newtheorem{lem}[thm]{Lemma}

\theoremstyle{definition}

\theoremstyle{remark}
\newtheorem{rem}[thm]{Remark}
\theoremstyle{definition}

\newcommand{\XX}{\mathbb{X}}
\newcommand{\PP}{\mathbb{P}}
\newcommand{\pr}{\mathbb{P}}

\begin{document}

\title{Plane curves containing a star configuration}

\author[E. Carlini]{Enrico Carlini}
\address[E. Carlini]{DISMA, Department of Mathematical Sciences, Poltiecnico di Torino, Turin, taly}
\email{enrico.carlini@polito.it}
\address[E. Carlini]{School of Mathematical Sciences, Monash University, Clayton, Australia}
\email{enrico.carlini@monash.edu}

\author[E. Guardo]{Elena Guardo}
\address[E. Guardo]{Dipartimento di Matematica,
Universita di Catania, Catania, Italy}
\email{guardo@dmi.unict.it}

\author[A. Van Tuyl]{Adam Van Tuyl}
\address[A. Van Tuyl]{Department of Mathematical Sciences,
Lakehead University, Thunder Bay, ON, Canada, P7B 5E1}
\email{avantuyl@lakeheadu.ca}

%%% ---------------------------------------------------------------

%\thanks{Version: \red{October 1, 2014}}

%%% ---------------------------------------------------------------

\begin{abstract}
Given a collection of $l$ general lines $\ell_1,\ldots,\ell_{l}$
in $\pr^2$, the star configuration $\XX(l)$
is the set of points constructed from all pairwise intersections of
these lines.
For each non-negative integer $d$, we compute the dimension of the
family of curves of degree $d$ that contain a star configuration.
\end{abstract}

%%% ----------------------------------------------------------------------
\maketitle
%%% --------------------------------------------------------------------

\section{Introduction}

Throughout this paper, $S = k[x_0,x_1,x_2]$ with $k$ an
algebraically closed field.  Given any linear form $L \in S$, we
let $\ell$ denote the corresponding line in $\pr^2$. 
{{Let
$\ell_1,\ldots,\ell_{l}$ be a set of $l$ distinct lines in
$\pr^2$.  Classically, the union of these $l$ lines is called an
{\it  $l$-lateral}. A {\it complete $l$-lateral} is the union of 
$l$ lines, such that $\ell_i \cap \ell_j \cap \ell_k = \emptyset$ for all triples $\{i,j,k\}
\subseteq \{1,\ldots,l\}$.  We say that a plane curve has an inscribed
$l$-lateral if it contains the $\binom{l}{2}$ vertices of the
$l$-lateral, that is, the $\binom{l}{2}$ points formed
by taking all possible intersections of lines.   

Let $\XX(l)$ denote the collection of points formed by taking all possible
intersections of a complete $l$-lateral.  Such a collection of points
is sometimes called a {\it star configuration}.}
%Given a collection of $l$ linear forms $L_1,\ldots,L_{l}$ in $S$
%such that $\ell_i \cap \ell_j \cap \ell_k = \emptyset$ for all
%triples $\{i,j,k\} \subseteq \{1,\ldots,l\}$, the {\it star
%configuration} $\XX(l)$ is the set of $\binom{l}{2}$ points formed
%by taking all possible intersections of lines.
The name star configuration arises from the fact that a {{complete $5$-lateral}
that contains an $\XX(5)$ resembles a star.  These special
configurations, and their generalizations in $\pr^n$, have recently risen
in prominence due, in part, to the fact that they have nice
algebraic properties (e.g., the minimal generators are products of
linear forms), but at the same time exhibit some extremal
properties (e.g., the work of Bocci and Harbourne \cite{BH} which
compares symbolic and regular powers of ideals). The papers
\cite{AS,BCH,FHL,GHM,GMS,PS,S} are some of the papers that have
contributed to our understanding of star configurations.
{Because this paper is
related to our previous papers (see \cite{CGVT,CVT}) on star configurations, we shall
prefer to use the terminology of star configurations as opposed
to the language of $l$-laterals found in \cite{D,ottaviani sernesi, ottaviani sernesi2}.
Moreover, star configurations may better lend themselves to higher dimensional generalizations (see
our concluding remarks).  Moving forward, we will primarily refer
to the family of curves in
$\pr^2$ of degree $d$ that ``contain a star configuration
$\XX(l)$''  as opposed to  ``contain an inscribed $l$-lateral''.}

In this paper we compute the dimension of the family of curves in
$\pr^2$ of degree $d$ that contain a star configuration $\XX(l)$,
{{or equivalently, an inscribed $l$-lateral.}
More precisely, if $l> 2$, consider the quasi-projective variety
\[\mathcal{D}_l \subseteq \underbrace{\check{\pr}^2 \times \cdots \times
\check{\pr}^2}_l\]
where $(\ell_1,\ldots,\ell_l) \in \mathcal{D}_l$ if and only if
no three of the lines meet at a point;  here $\check{\pr}^2$ denotes
the dual projective space. Notice that $\mathcal{D}_l$ can be seen as a 
parameter space for star configuration set of points obtained by intersecting $l$ 
general lines. With a slight abuse of notation, we will often write $\XX(l)\in\mathcal{D}_l$, 
thus identifying a star configuration with the unique set of lines defining it.

We construct the following incidence correspondence
\begin{equation}\label{correspondence}
\Sigma_{d,l} = \{(\mathcal{C},\XX(l)) :  \mathcal{C} \supseteq \XX(l) \} \subseteq
\PP S_d \times \mathcal{D}_l. \tag{$\star$}
\end{equation}
Letting $\phi_{d,l}:\Sigma_{d,l} \rightarrow \PP S_d$ denote the natural
projection map, we define the {\it locus of degree $d$ curves containing
a  star configuration $\XX(l)$}, denoted $\mathcal{S}(d,l)$, to be
$\mathcal{S}(d,l) = \phi_{d,l}(\Sigma_{d,l})$.  We then prove the following
result about the dimension of the locus.

\begin{thm}\label{maintheorem} Let $d \geq 0$ and $l \geq 2$ be integers.  Then
$\mathcal{S}(d,l) = \emptyset$ if $d < l-1$, and
\[\dim \mathcal{S}(d,l) = \left\{
\begin{array}{ll}
\binom{d+2}{2}-1 & \mbox{if $d \geq l-1$ and $l=2,3,4$}\\
\binom{d+2}{2}-2 & \mbox{if $d =4$ and $l = 5$} \\
\binom{d+2}{2}-1 & \mbox{if $d \geq 5$ and $l =5$} \\
 \binom{d+2}{2}-\binom{l}{2}+2l-1 & \mbox{if $d \geq l-1$ and $l \geq 6$.}
\end{array}
\right.\]
\end{thm}

Theorem \ref{maintheorem} complements our previous work
\cite{CGVT,CVT} which showed that the generic degree $d$ plane
curve contains a star configuration $\XX(l)$ if and only if the
projection map $\phi_{d,l}$ is dominant which happens if and only
if $\dim \mathcal{S}(d,l) = \binom{d+2}{2}-1$.  The tuples $(d,l)$
for which $\dim \mathcal{S}(d,l) = \binom{d+2}{2}-1$ are therefore
precisely the tuples described in \cite[Theorem 6.3]{CVT}. {
{The fact that $\mathcal{S}(d,l)= \emptyset$ for $d<l-1$ comes
from Bezout's Theorem {(see Remark \ref{dimension0})}.}}
%Note that $\mathcal{S}(d,l) = \emptyset$ if no curve of degree
%$d$ contains a star configuration $\XX(l)$.  Because the defining
%ideal of $\XX(l)$ is minimal generated in degree $l-1$, it
%follows that $\mathcal{S}(d,l) = \emptyset$ for $d < l-1$.

It therefore suffices to focus on proving Theorem
\ref{maintheorem} for the pairs $(4,5)$ and $(d,l)$ with $d \geq
l-1$ and $l \geq 6$.  Our strategy for the pairs $(d,l) \neq
(4,5)$ is to first translate the problem into computing the
dimension of a graded ideal constructed from the linear forms
$L_1,\ldots,L_l$ in a particular degree.  This enables us to
reduce the problem to computing the rank of a particular matrix.
We use the notion of L\"uroth quartics to deal with the pair
$(d,l) = (4,5)$. {{Note that $\mathcal{S}(4,5)$ is
the only $\mathcal{S}(d,l)$ whose dimension differs from the
expected one.

The family $\mathcal{S}(d,d+1)$ was also studied by Barth \cite[Application 2]{Ba}, 
who also computed their dimensions, and by 
Ellingsrud, Le Potier, and Str\o mme \cite[Section 4]{ELS}, who
raised the still-open question of computing the cardinality of the fibres of 
$\phi_{d,d+1}$.}}

Our paper is structured as follows.  In Section 2 we recall the
relevant facts about star configurations. We also translate our
problem into a new algebraic question, and we compute $\dim
\mathcal{S}(4,5)$. In Section 3 we prove Theorem \ref{maintheorem}
for all tuples $(d,6)$ with $d \geq 5$.   The results of this
section provide a base case for the arguments of Section 4.  We
conclude with remarks about the higher dimensional analog of this
problem.

\noindent {\bf Acknowledgements}  The computer algebra system
CoCoA \cite{C} played an integral role in this project. The first
author thanks the Universit\`a di Catania for its hospitality and
the Politecnico di Torino for financial support.  The third author
acknowledges the support of NSERC. We would like to thank G.
Ottaviani for suggesting this problem to us.  { {We also thank
the referee for his/her useful suggestions and improvements.}}

%%%%%%%%%%%%%%%%%%%%%%%%%%%%%%%%%%%%%%%%%%%%%%%%%%%%%%%%%%%%%%%%%%%%%%

\section{Properties of star configurations}\label{section2}

We recall the relevant results about star configurations in $\pr^2$
and prove an upper bound on $\dim \mathcal{S}(d,l)$.
Although some proofs are omitted, they can
be found in \cite{CGVT,GHM}.

We continue to use the notation
introduced in the introduction.
For any $l \geq 2$,  let $L_1,\ldots,L_l$ be a collection
of $l$ linear forms of $S = k[x_0,x_1,x_2]$ that are
three-wise linearly independent.  We call
such a collection a collection of  {\it general linear forms}.
We let $\XX(l)$
denote the star configuration of $\binom{l}{2}$ points in $\pr^2$
which is formed from all pairwise intersections of the $l$ linear forms.
Note that when $l=2$, then $\XX(2)$ is
simply a point, and if $l =3$, then $\XX(3)$ is three non-collinear points.

%{\red {We first observe the following result from Bezout's
%Theorem:
%\begin{prop} \label{emptyset}$\mathcal{S}(d,l)= \emptyset$ for
%$d<l-1$.
%\end{prop}
%\begin{proof} If a curve lies in $\mathcal{S}(d,l)$, every line of the $l$-lateral meets
%the curve in the $(l-1)$ vertices obtained by the intersection of
%this line with the others $(l-1)$, hence it is a component of the
%curve. A degree $d$ curve cannot contain $l$ lines, hence the
%result.
%\end{proof}}

The following lemma allows us to describe the minimal generators
of the ideal associated to $\XX(l)$ and the Hilbert function of this
ideal.

\begin{lem}\label{basicprop}
 Let $L_1,\ldots,L_l$ be $l\geq 2$ general linear forms of
 $S = k[x_0,x_1,x_2]$, and let $I_{\XX(l)}$ denote
the defining ideal of $\XX(l)$.
\begin{enumerate}
\item[$(i)$]  For each $i \in \{1,\ldots,l\}$, let
$\hat{L}_i = \prod_{j\neq i} L_j$.  Then
$I_{\XX(l)} = (\hat{L}_1,\hat{L}_2,\ldots,\hat{L}_l)$.
\item[$(ii)$]  The set of points $\mathbb{X}(l)$ has the
Hilbert function of  ${l\choose 2}$ generic points, that is
\[HF(\mathbb{X}(l),t)= \dim_k (S/I_{\XX(l)})_t =
 \min\left\{{t+2\choose 2},{l\choose 2}\right\}
~~\mbox{for $t \geq 0$.}\]
\end{enumerate}
 \end{lem}

\begin{proof}
For $(i)$, see \cite[Lemma 2.3(iv)]{CGVT}.  For $(ii)$, see
\cite[Theorem 2.5]{CGVT}.
 \end{proof}

\begin{rem}\label{dimension0}
Lemma \ref{basicprop} $(i)$ shows that no plane curve of
degree $d$ with $d < l-1$ can contain a star configuration $\XX(l)$.
As a consequence, $\mathcal{S}(d,l) = \emptyset$ for all $(d,l)$ with $d < l-1$.
Alternatively, one could use Bezout's theorem:  if  a degree $d < l-1$ curve  {is} in 
$\mathcal{S}(d,l)$, every line of the star configuration meets
the curve at the $(l-1)$ points obtained by the intersection of
this line with the other $(l-1)$ lines, hence it is a component of the
curve. { A degree $d$ curve cannot contain $l>d+1$ lines}, hence the result.
\end{rem}

{
\begin{rem}
Toh\v{a}neanu \cite{T} 
also considers ideals generated by products of linear forms
with a connection to coding theory.  When the linear forms are general, his
results give an alternative proof to Lemma \ref{basicprop}.
\end{rem}}

We give an upper bound on
$\dim \mathcal{S}(d,l)$.

\begin{lem} \label{upperbound}
Let $l \geq 2$ and $d \geq l-1$ be integers.
Then
\[\dim \mathcal{S}(d,l) \leq  \binom{d+2}{2}-\binom{l}{2}+2l-1.\]
\end{lem}

\begin{proof}
Consider the incidence correspondence $\Sigma_{d,l}$ as given in
\eqref{correspondence}, and let
\[\psi_{d,l}:\Sigma_{d,l}\longrightarrow \mathcal{D}_l
 ~~\text{and}~~~ \phi_{d,l}:\Sigma_{d,l}\longrightarrow \PP S_d\]
be the natural projection maps.
 Note that we are following
the standard convention that $\pr S_d$ is identified with the projective
space $\pr^{N_d}$ where $N_d = \binom{d+2}{2}-1$. Using a standard fibre dimension argument,
if $d\geq l-1$, then
 \[\dim\Sigma_{d,l}\leq \dim \mathcal{D}_l +
 \dim_{k} (I_{\mathbb{X}(l)})_d -1 =
 \dim\mathcal{D}_l+{2+d\choose d}-{l\choose 2} -1= 2l+{2+d\choose d}-{l\choose 2} -1.\]
Here, we are using Lemma \ref{basicprop} $(ii)$ since
$\dim_k S_d - \dim_k (I_{\XX(l)})_d = \binom{l}{2}$ when $d \geq l-1$.
The desired bound now follows from the fact that $\dim \mathcal{S}(d,l) \leq \dim
\Sigma_{d,l}$.
\end{proof}

{\begin{rem} In \cite{Ba}, Section 5, Lemma 2 or in
\cite{D}, Lemma 6.3.24, the authors compute the dimension of
$(I_{\XX(l)})_{l+1}$ and this result can be used to alternatively prove Lemma \ref{upperbound}.
\end{rem}}}

\begin{rem}
As shown in \cite[Theorem 3.1]{CVT}, if $l\geq 6$, then the map
$\phi_{d,l}$ cannot be dominant.
\end{rem}

Inspired by our previous work \cite{CGVT,CVT}, we can reformulate
the problem of computing $\dim \mathcal{S}(d,l)$ in terms of
computing the dimension of an ideal in a specific degree.  In fact,
the proof of \cite[Lemma 4.3]{CVT} already contains the result we
need.
We first perform the following geometric
construction. With $d \geq l -1$ define a map of affine varieties
\[\Phi_{d,l}:\underbrace{ S_1\times\cdots\times
S_1}_{l}\times\underbrace{ S_{d-l+1}\times\cdots\times
S_{d-l+1}}_{l}\longrightarrow S_d\] such that
\[\Phi_{d,l}\left(L_1,\ldots,L_l,M_1,\ldots,M_l\right)=\sum_{i=1}^l M_i\hat
L_i.\] We then rephrase our problem in terms of the map
$\Phi_{d,l}$.

\begin{lem} With notation as above, the image of $\Phi_{d,l}$ is the affine cone
over $\mathcal{S}(d,l)$.  In particular,
$\dim \mathcal{S}(d,l) =
\dim \operatorname{Im}(\Phi_{d,l})-1$.
\end{lem}

\begin{proof}
Suppose $H$ is a degree $d$ form that defines a curve
$\mathcal{C}$ that contains a star
configuration $\XX(l)$.  So, there exists linear forms
$L_1,\ldots,L_l$ such that $H \in (\hat{L}_1,\ldots,\hat{L}_l)$,
and hence $H = \sum_{i=1}^l M_i\hat{L}_i$ with $M_i \in S_{d-l+1}$ for each
$i$.  But this means that $H \in \operatorname{Im}(\Phi_{d,l})$.
Viewing elements of $\mathcal{S}(d,l)$ as elements of $\pr S_d$, this
gives $\dim \mathcal{S}(d,l) \leq \dim \operatorname{Im}(\Phi_{d,l})-1$.

Now consider a generic $F \in \operatorname{Im}(\Phi_{d,l})$.
We want
to show that there exists $$(L_1,\ldots,L_l,M_1,\ldots,M_l) \in
\Phi^{-1}_{d,l}(F)$$
such that the linear forms define a star configuration.
Define
$\Delta\subset  S_1\times\cdots\times S_1\times
S_{d-l+1}\times\cdots\times S_{d-l+1}$ as follows:
\[\Delta = \left\{\left(L_1,\ldots,L_l,M_1,\ldots,M_l\right)~\left|
\begin{tabular}{l}
\mbox{there exists $a \neq b \neq c$ such that} \\
\mbox{$L_a,L_b,L_c$ are linearly dependent}
\end{tabular} \right\}\right. .\]
It suffices to show that $\Phi_{d,l}^{-1}(F)$ is
not contained in
$\Delta$ for a generic form $F \in \operatorname{Im}(\Phi_{d,l})$.
Suppose for a contradiction that $\Phi^{-1}_{d,l}(F)$ is contained in $\Delta$.
Then $\Delta$ would be a component of the domain of $\Phi_{d,l}$.
This is a contradiction as the latter is an irreducible variety
being the product of
irreducible varieties.  This completes the proof.
\end{proof}

As a consequence of the above lemma, we only need to compute
$\dim \operatorname{Im}(\Phi_{d,l})$.  As we now show, we can compute
$\dim \operatorname{Im}(\Phi_{d,l})$ by the size of its tangent space.
In fact, this value will equal the vector space dimension
of a graded ideal in a specific degree.

\begin{lem}\label{tgspaceideallem}
Let $l \geq 2$ and $d \geq l-1$ be integers, and consider
$l$ general linear forms $L_1,\ldots,L_l$ in $S$.   Also,
let $M_1,\ldots,M_l \in S_{d-l+1}$ be any homogeneous forms of degree
$d-l+1$.  Set
\[\hat{L}_i = \prod_{j\neq i} L_j ~~\text{and}~~ \hat{L}_{i,j} =
\prod_{h \not\in \{i,j\}} L_h ~~\mbox{for $i\neq j$}.\]
Define the following $l$ forms of degree $d-1$:
\begin{eqnarray*}
Q_1 &= & M_2\hat{L}_{1,2} + M_3\hat{L}_{1,3} + \cdots M_l\hat{L}_{1,l} =
\sum_{i \neq 1} M_i\hat{L}_{1,i} \\
Q_2 &= &  M_1\hat{L}_{2,1} + M_3\hat{L}_{2,3} + \cdots M_l\hat{L}_{2,l} =
\sum_{i \neq 2} M_i\hat{L}_{2,i} \\
& \vdots & \\
Q_l &= &  M_1\hat{L}_{l,1} + M_2\hat{L}_{l,2} + \cdots M_{l-1}\hat{L}_{l,l-1} =
\sum_{i \neq l} M_i\hat{L}_{l,i}. \\
\end{eqnarray*}
With this notation, form the ideal
\[I= (\hat{L}_1,\cdots,\hat{L}_l) + (Q_1,\ldots,Q_l) = I_{\XX(l)} + (Q_1,\ldots,Q_l) \subseteq S.\]
Then $I_d$ is the affine tangent
space to $\operatorname{Im}(\Phi_{d,l})$ at a generic point.  In particular,
\[\dim \mathcal{S}(d,l) = \dim_k I_d -1.  \]
\end{lem}

\begin{proof}
The statement about $\dim \mathcal{S}(d,l)$ follows from the first statement
and the previous lemma.  We need to
determine the tangent space $\operatorname{Im}(\Phi_{d,l})$ in a
generic point $q=\Phi_{d,l}(p)$, where
$p=(L_1,\ldots,L_l,M_1,\ldots,M_l)$. We denote with $T_q$ this
affine tangent space.

The elements of the tangent space $T_q$ are obtained as

\[\left.{d \over dt}\right|_{t=0}
\Phi_{d,l}\left(L_1+tL_1',\ldots,L_l+tL_l',M_1+tM_1',\ldots,M_l+tM_l'\right)\]
when we vary the forms $L_i'\in S_1$ and $M_i'\in S_{d-l+1}$. By a
direct computation we see that the elements of $T_q$ have the form
\[M_1'\hat L_1+\cdots+M_l'\hat L_l+ L_1'(M_2\hat
L_{1,2}+\cdots+M_l\hat L_{1,l}) +\cdots+ \]
\[+L_j'(M_1\hat L_{j,1}+\cdots+M_l\hat L_{j,l})+\cdots+L_l'(M_1\hat
L_{l,2}+\cdots+M_{l-1}\hat L_{l,l-1}),\] where $\hat
L_i=\prod_{j\neq i}L_j$ and $\hat L_{i,j}=\prod_{h \not\in \{i,j\}}L_h, \mbox{ for } i\neq j$.

Since the $L_i'\in S_1$ and $M_i'\in S_{d-l+1}$ can be chosen
freely, we obtain $I_d=T_q$.
\end{proof}

\begin{rem}\label{uppersemi}
As in \cite{CVT}, we can use the above lemma and appeal
to upper-semicontinuity to compute $\dim \mathcal{S}(d,l)$ if we know
a (good) upper bound on $\dim \mathcal{S}(d,l)$.  Indeed, suppose
we know that $\dim \mathcal{S}(d,l) \leq M$.
Given $d$ and $l$ we construct the ideal $I$ as
in Lemma \ref{tgspaceideallem} by choosing forms $L_i$ and $M_i$.
Then we compute
$\dim_k I_d$ using a computer algebra system. If
$\dim_k I_d - 1 = M$, by upper semi-continuity of the
dimension (indeed, the dimension can
decrease only on a proper closed subset),
we have proved $M= \dim_k I_d - 1 \leq \dim \mathcal{S}(d,l) \leq M$, and hence $\dim \mathcal{S}(d,l)=M$
for this pair $(d,l)$.  We will require
this technique for some small values of $d$ and $l$.
\end{rem}

%{\red {In his paper \cite{lurothcite}, L\"uroth proved that a
%nonsingular quartic plane curve containing the ten vertices of a
%complete pentalateral contains infinitely many such $10$-tuples.
%This implies that such curves, called L\"uroth quartics, fill an
%open set of an irreducible, $SL(3)$-invariant hypersurface $
%\mathcal{L}\subset \pr^{14}$.}}

{ {Using our notation,}} it is known that the generic plane
quartic does not contain a $\XX(5)$. We call a quartic containing
a $\XX(5)$ a {\it L\"uroth quartic}.  These objects were
classically studied; see for example
\cite{Ba,D,lurothcite,morley}, and for a modern treatment
\cite{ottaviani sernesi, ottaviani sernesi2}.  Of interest is the
following theorem.

\begin{thm}[{\cite[Theorem 11.4]{ottaviani sernesi}}]\label{luroth}
{\it L\"{u}roth quartics form a hypersurface of degree $54$ in the
space of plane quartics.}
\end{thm}

We can now prove $\dim \mathcal{S}(4,5) = 13$ using our techniques.

\begin{rem}\label{lurothbound}  By Lemma \ref{upperbound},
$\dim \mathcal{S}(4,5) \leq \binom{6}{4}-\binom{5}{2}+2\cdot5-1
= 14$.  However, Theorem \ref{luroth} implies that the projection
map $\phi_{d,l}: \Sigma_{d,l} \rightarrow \pr S_d$ is not dominant,
so $\dim \mathcal{S}(4,5) \leq 13$.

Now consider the following five linear forms
\[L_1 = x_0; ~~ L_2 = x_1;~~  L_3 = x_2;  ~~ L_4 = x_0+x_1+x_2; ~~\text{and}
~~ L_5 = x_0+2x_1 + 3x_3.\]
We construct the ideal $I$ as in Lemma \ref{tgspaceideallem} where
we take $M_1 = \cdots = M_5 = 1$.  Using
CoCoA\footnote{For our code, see
{\tt http://flash.lakeheadu.ca/$\sim$avantuyl/research/PlaneCurvesStarConfig\_CGVT.html}}
we find that
$\dim_k (I_d) = 14$, whence $13 = \dim_k (I_d) -1 \leq
\dim \mathcal{S}(4,5) \leq 13$,
thus giving the desired result via Remark \ref{uppersemi}.
\end{rem}

Some additional
remarks about computing $\dim_k I_d$ are given below.

\begin{rem} \label{strategy}
 Consider the ideal $I$ in Lemma \ref{tgspaceideallem}.
We wish to compute $\dim_k I_d$.  Now $I = I_{\XX(l)} + (Q_1,\ldots,Q_l)
\subseteq S$.  Because $d \geq l-1$,  by Lemma \ref{basicprop}
$\dim_k(I_{\XX(l)})_d = \binom{2+d}{2}-\binom{l}{2}$.
It then follows that to compute $\dim_k I_d$, it
is enough to compute $\dim_k (\overline{Q}_1,\ldots,\overline{Q}_l)_d$
where
$(\overline{Q}_1,\ldots,\overline{Q}_l) = I/I_{\XX(l)}
\subseteq S/I_{\XX(l)}.$
So we have
\[\dim \mathcal{S}(d,l) = \dim_k (I_d) - 1 =
\dim_k (\overline{Q}_1,\ldots,\overline{Q}_l)_d + \binom{2+d}{2}-\binom{l}{2}
-1.\]
By Lemma \ref{upperbound}, we know $\dim \mathcal{S}(d,l) \leq
\binom{2+d}{2}-\binom{l}{2} + 2l-1.$   So, if we can show that
$\dim_k (\overline{Q}_1,\ldots,\overline{Q}_l) = 2l$ for a specific choice
of $Q_i$'s, then by Remark \ref{uppersemi}, we will in fact have the
equality
$\dim \mathcal{S}(d,l) =
\binom{2+d}{2}-\binom{l}{2} + 2l-1.$

We will employ the following strategy in Sections 3 and 4.
After fixing some star configuration $\XX(l)$, we identify
$2l$ points $p_i$ in the star configuration, and determine $2l$ linear
forms $H_1,\ldots,H_{2l}$.  We then construct a $2l \times 2l$ evaluation
matrix
\[
\begin{array}{c|ccccc}
       & H_1Q_1& H_2Q_1& H_3Q _2& \cdots      & H_{2l}Q_l\\
\hline
p_{1}   & \delta_{1,1}     &  \delta_{1,2}    &  \delta_{1,3}       &  \cdots     &   \delta_{1,2l}      \\
p_{2}   &  \delta_{2,1}     &  \delta_{2,2}    &  \delta_{2,3}       &  \cdots     &   \delta_{2,2l}      \\
p_{2}   &  \delta_{3,1}     &  \delta_{3,2}    &  \delta_{3,3}       &  \cdots     &   \delta_{3,2l}      \\
\vdots &       &        &       &  \vdots     &  \\
p_{2l}  & \delta_{2l,1}     &  \delta_{2l,2}    &  \delta_{2l,3}       &  \cdots     &   \delta_{2l,2l}      \\
\end{array}
\]
where $\delta_{i,j}$ is the point $p_i$ evaluated at the degree $d$ form
that indexes column $j$.
If $\mathcal{M}$ denotes the resulting matrix, then $\mathrm{rk}(\mathcal{M}) =
\dim_k (\overline{Q}_1,\ldots,\overline{Q}_l)_d$ in $S/I_{\XX(l)}$.
\end{rem}

%%%%%%%%%%%%%%%%%%%%%%%%%%%%%%%%%%%%%%%%%%%%%%%%%%%%%%%%%%%%%%%%%%%%%%

\section{The case $l = 6$ and $d\geq 5$}

In this section we compute $\dim \mathcal{S}(d,6)$ for all $d \geq 6-1 = 5$.
We make use of the following notion:  if $\XX(l)$
is the star configuration constructed from $L_1,\ldots,L_l$,
we let $p_{i,j}$ with $1 \leq i < j \leq l$ denote
the point formed by intersection of $\ell_i$ and $\ell_j$.  Thus
$\XX(l) = \{p_{i,j} = \ell_i \cap \ell_j ~:~ 1 \leq i < j \leq l\}.$

\begin{thm}\label{l=6}
For all $d \geq 5$, $\dim \mathcal{S}(d,6) = \binom{d+2}{2}-4$.
\end{thm}

\begin{proof}
We break the proof into three cases:  (1) $d= 5$, (2) $d=6$, and (3)
$d \geq 7$.  For all three cases, we will use the strategy outlined
in Remark \ref{strategy};  thus, it suffices to construct a $12 \times 12$
evaluation matrix with maximal rank.

(1) For the case $d=5$, consider the six linear forms
\[L_1 = x_0; ~~ L_2 = x_1;~~  L_3 = x_2;  ~~ L_4 = x_0+x_1+x_2;
~~ L_5 = x_0+2x_1 + 3x_2; ~~\text{and}~~ L_6 = x_0+3x_1 + 10x_2.\]
When constructing the $Q_i$'s, we set $M_i = 1$ for $i=1,\ldots,6$.
We form the evaluation matrix with columns indexed by
\[L_2Q_4,L_1Q_4,L_3Q_5,L_2Q_5,L_1Q_6,L_3Q_6,L_6Q_1,L_3Q_2,L_6Q_2,L_6Q_3,L_4Q_3,L_5Q_1\]
and rows are index by
\[p_{1,4},p_{2,4},p_{2,5},p_{3,5},p_{3,5},p_{3,6},p_{2,6},p_{1,5},p_{2,6},p_{2,3},
p_{3,4},p_{1,2}.\]
We used CoCoA  to verify that the resulting matrix has the desired rank of 12 (our code can be found
on the third author's website).

(2) For the case $d=6$, we use the same $L_i$'s, but when we construct the $Q_i$'s, we first
find a linear form $G$ that does not contain any of the points of $\XX(6)$, and
set $M_i = G$ for $i=1,\ldots,6$.  Again, the resulting evaluation
matrix (using the same indexing for the rows and columns) has
maximal rank.

(3) We now consider the case $d \geq 7$.  Pick any six general linear
forms $L_1,\ldots,L_6$ that form a $\XX(6)$, and let
$p_{1,2},\ldots,p_{4,5}$ be the 15 points of $\XX(6)$. In order to
construct the $Q_i$'s as defined in Lemma \ref{tgspaceideallem}, we
need to pick six forms $M_1,\ldots,M_6$ in $S_{d-l+1}$.  Since $d
\geq 7$, each $M_i$ will have degree at least two.   We construct
the $M_i$'s as follows.  First, we pick six linear forms in $S$
with the following properties:
\begin{enumerate}
\item[$\bullet$] $G$ is a linear form such that the line $G=0$ does not
pass through any point of $\XX(6)$;
\item[$\bullet$] $G_1$ is a linear form such that line $G_1=0$ only passes
through the point $p_{1,5}$ of $\XX(6)$;
\item[$\bullet$] $G_2$ is a  linear form such that the line $G_2=0$ only passes
through the point $p_{1,2}$ of $\XX(6)$;
\item[$\bullet$] $G_3$ is a  linear form such that the line $G_3=0$ only passes
through the point $p_{2,6}$ of $\XX(6)$;
\item[$\bullet$] $G_4$ is a  linear form such that the line $G_4=0$ only passes
through the point $p_{3,4}$ of $\XX(6)$;
\item[$\bullet$] $G_5$ is a  linear form such that the line $G_5=0$ only passes
through the point $p_{4,6}$ of $\XX(6)$.
\end{enumerate}
We then set
\[\begin{array}{lll}
M_1 = G_1G_2G^{d-l-1} & M_2 = G_3G^{d-l} & M_3 = G_4G^{d-l} \\
M_4 = G^{d-l+1}      &M_5 = G^{d-l+1}  & M_6 = G_5G^{d-l}
\end{array}\]
and use these $M_i$'s to construct $Q_1,\ldots,Q_6$.

We now consider the $12 \times 12$ evaluation matrix (see below) whose columns and
rows are also indexed as above.  When determining the entries of the evaluation matrix, note
that
$L_rQ_t(p_{i,j}) = 0 $ if $i=r$, or $j=r$, or neither $i$ and $j$ equal $t$.
We will also have
\[\begin{array}{lllll}
L_3Q_5(p_{1,5}) = 0 & L_1Q_6(p_{2,6}) = 0 & L_2Q_4(p_{3,4}) = 0 & L_2Q_4(p_{4,6})=0
& L_3Q_2(p_{1,2})=0 \\
L_2Q_5(p_{1,5}) = 0 & L_3Q_6(p_{2,6}) = 0 & L_1Q_4(p_{3,4}) = 0 & L_1Q_4(p_{4,6})=0
& L_6Q_2(p_{1,2}) =0.
\end{array}\]
This follows from our choice of $M_i$'s.  For example,
\footnotesize
\begin{eqnarray*}
L_3Q_5(p_{1,5}) & = & L_3(M_1L_2L_3L_4L_6+M_2L_1L_3L_4L_6+M_3L_1L_2L_4L_6+
M_4L_1L_2L_3L_6 + M_6L_1L_2L_3L_4)(p_{1,5}) \\
& = & M_1L_2L_3^2L_4L_6(p_{1,5}) = 0
\end{eqnarray*}
\normalsize
since $M_1(p_{1,5}) = 0$.  Again by our choice of $M_i$'s, we have
\[
\begin{array}{lllll}
L_2Q_4(p_{1,4}) \neq 0 & L_1Q_4(p_{2,4}) \neq 0 & L_3Q_5(p_{2,5}) \neq 0
& L_2Q_5(p_{3,5}) \neq 0 & L_1Q_6(p_{3,6}) \neq 0 \\
L_3Q_3(p_{1,6}) \neq 0 & L_4Q_3(p_{3,6}) \neq 0 & L_5Q_1(p_{1,6}) \neq 0 &
L_1Q_6(p_{4,6}) \neq 0 & L_3Q_6(p_{4,6}) \neq 0 \\
&L_6Q_1(p_{1,2}) \neq 0 & L_5Q_1(p_{1,2}) \neq 0.
\end{array}
\]
For example,
\footnotesize
\begin{eqnarray*}
L_5Q_1(p_{1,2}) &=& L_5(M_2L_3L_4L_5L_6+M_3L_2L_4L_5L_6+M_4L_2L_3L_5L_6
+ M_5L_2L_3L_4L_6 + M_6L_2L_3L_4L_5)(p_{1,2}) \\
& = & M_2L_3L_4L_5^2L_6(p_{1,2}) \neq 0
\end{eqnarray*}
\normalsize
since $M_2$ does not vanish at $p_{1,2}$, and $p_{1,2}$ does not lie
on the lines defined by $L_3,L_4,L_5$ or $L_6$.

Our evaluation matrix therefore has the form
\footnotesize{
\[
\begin{array}{c|cccccccccccc}
       & L_2Q_4 & L_1Q_4 & L_3Q_5 & L_2Q_5 & L_1Q_6 & L_3Q_6 & L_6Q_1 & L_3Q_2 & L_6Q_2 & L_6Q_3 & L_4Q_3 & L_5Q_1 \\
\hline
p_{1,4} & \star & 0     & 0      & 0      & 0     &  0    & \Box  & 0      & 0     & 0      &  0     & \Box  \\
p_{2,4} & 0     & \star & 0      & 0      &0      &  0    & 0     &\Box    &\Box   & 0      & 0      & 0  \\
p_{2,5} & 0     & 0     &\star   & 0      &0      &  0    & 0     &0       & \Box  &  0     & 0      & 0  \\
p_{3,5} & 0     & 0     & 0      & \star  &0      &  0    & 0     &0       & 0  &  \Box  & \Box      & 0  \\
p_{3,6} & 0     & 0     & 0      & 0      &\star  &  0    & 0     &0       & 0     &  0     & \star  & 0  \\
p_{1,6} & 0     & 0     & 0      & 0      &0      &\star  & 0     &0       & 0     &  0     & 0      & \star  \\
p_{1,5} & 0     & 0     & 0      & 0      & 0     &  0    & \star & 0      & 0     & 0      & 0      & 0  \\
p_{2,6} & 0     & 0     & 0      & 0      &0      &  0    & 0     &\star   & 0     & 0      & 0      & 0  \\
p_{2,3} & 0     & 0     & 0      & 0      &0      &  0    & 0     &0       & \star & \star  & \star  & 0  \\
p_{3,4} & 0     & 0     & 0      & 0      &0      &  0    & 0     &0       & 0     &  \star & 0      & 0  \\
p_{4,6} & 0     & 0     & 0      & 0      &\star  &  \star& 0     &0       & 0     &  0     & 0      & 0  \\
p_{1,2} & 0     & 0     & 0      & 0      &0      &  0    & \star &0       & 0     &  0     & 0      & \star  \\
\end{array}
\]}
\normalsize

\noindent where $\star$ denotes a nonzero entry and $\Box$ denotes
an entry which may or may not be zero. Using Gaussian elimination,
we get a matrix in row echelon form (where the nonzero leading
coefficients are not necessarily equal to 1), with zero entries
below the diagonal. Consequently, the original matrix has maximal
rank, as desired.
\end{proof}

%%%%%%%%%%%%%%%%%%%%%%%%%%%%%%%%%%%%%%%%%%%%%%%%%%%%%%%%%%%%%%%%%%%%%%

\section{The case $l > 6$ and $d \geq l-1$}

We now evaluate $\dim \mathcal{S}(d,l)$ for all $l \geq 6$ when $d \geq l-1$.
The key idea is to pick the $l$ linear forms $L_1,\ldots,L_l$
that define $\XX(l)$ so that the first six forms are
as in Theorem \ref{l=6}.

\begin{thm}\label{l>6}
Let $l \geq 7$ and $d \geq l-1$.  Then
$\dim \mathcal{S}(d,l) = \binom{d+2}{2}-\binom{l}{2}+2l-1$.
\end{thm}

\begin{proof}
As in Theorem \ref{l=6}, it suffices to construct a $2l \times 2l$
evaluation matrix of rank $2l$.    Let $L_1,\ldots,L_l$ be the $l$
general linear forms that define $\XX(l)$.  If $d=l-1$ or $d=l$, we
let $L_1,\ldots,L_6$ be as in Theorem \ref{l=6}.  When constructing the
$Q_i$'s, we use the following $M_i$'s:
\begin{enumerate}
\item[$\bullet$] If $d=l-1$, let $M_i=1$ for $i=1,\ldots,l$.
\item[$\bullet$] If $d=l$, let $M_i = G$, where $G$ is a linear form
such that the curve $G=0$ does not contain any of the points of $\XX(l)$.
\item[$\bullet$] If $d\geq l+1$, define $M_1,\ldots,M_6$ as in
Theorem \ref{l=6}, but with the added condition that each $M_i$ also
does not vanish at any other point of $\XX(l)$.
We set  $M_7 = \cdots = M_l = G^{d-l+1}$, where again $G$ is a linear
form such that the curve $G=0$ does not contain any points of $\XX(l)$.
\end{enumerate}

When we form our evaluation matrix, we label the first twelve columns
as in Theorem \ref{l=6}, and we label the remain $2l-12$ columns with
$L_2Q_7,L_1Q_7,L_2Q_8,L_1Q_8,\ldots,L_2Q_l,L_1Q_l.$
We label the first twelve rows as in Theorem \ref{l=6} and the
remaining rows are labelled with
$p_{1,7},p_{2,7},p_{1,8},p_{2,8},\ldots,p_{1,l},p_{2,l}.$
Our evaluation matrix then has the form:
\footnotesize{
\[
\begin{array}{c|ccccc|ccccccc}
       & L_2Q_4 & L_1Q_4 & \cdots & L_4Q_3& L_5Q_1& L_2Q_7 & L_1Q_7& L_2Q_8  & L_1Q_8& \cdots & L_2Q_l  & L_1Q_l\\
\hline
p_{1,4} & \star & 0     & \cdots  &  0     & \Box &  0    & 0     & 0      & 0     & \cdots  & 0     & 0     \\
p_{2,4} & 0     & \star & \cdots  & 0      & 0    &  0    & 0     & 0      & 0     & \cdots  & 0     & 0     \\
\vdots &       &       & \vdots  &        &      &       &       &        &       & \vdots  & 0     & 0     \\
p_{4,6} & 0     & 0     & \cdots  & 0      & 0    & 0     & 0     & 0      & 0     & \cdots  & 0     & 0     \\
p_{1,2} & 0     & 0     & \cdots  & 0      & \star& 0     & 0     & 0      & 0     & \cdots  & 0     & 0     \\
\hline
p_{1,7} & 0     & 0     & \cdots  & 0      & \Box & \star & 0     & 0      & 0     & \cdots  & 0     & 0     \\
p_{2,7} & 0     & 0     & \cdots  & 0      & \Box & 0     & \star & 0      & 0     & \cdots  & 0     & 0     \\
p_{1,8} & 0     & 0     & \cdots  & 0      & \Box & 0     & 0     & \star  & 0     & \cdots  & 0     & 0     \\
p_{2,8} & 0     & 0     & \cdots  & 0      & \Box & 0     & 0     & 0      & \star & \cdots  & 0     & 0     \\
\vdots &       &       & \vdots  &        &      &       &       &        &       & \vdots  & 0     & 0     \\
p_{1,l} & 0     & 0     & \cdots  & 0      & \Box & 0     & 0     & 0      & 0     & \cdots  & \star & 0     \\
p_{2,l} & 0     & 0     & \cdots  & 0      & \Box & 0     & 0     & 0      & 0     & \cdots  & 0     & \star  \\
\end{array}
\]}
\normalsize

\noindent
where $\star$ denotes a nonzero entry and $\Box$ denotes an entry which may or may not be zero.

Consider the $12 \times 12$ sub-matrix formed by the first $12$ rows and $12$ columns.  Let $Q'_i$ denote
the form constructed as in Lemma \ref{tgspaceideallem} using $L_1,\ldots,L_6$ and
the same $M_i$'s as above.  Then, for every nonzero
entry in this sub-matrix, we have
\[L_rQ_t(p_{i,j}) = L_rQ'_tL_7L_8\cdots L_l(p_{i,j})= [L_rQ'_t(p_{i,j})][L_7L_8\cdots L_l(p_{i,j})].\]
For example
\begin{eqnarray*}
L_2Q_4(p_{1,4}) & =& L_2(M_1\hat{L}_{4,1} + M_2\hat{L}_{4,2} + \cdots M_l\hat{L}_{4,l})(p_{1,4}) =  L_2M_1\hat{L}_{4,1}(p_{1,4}) \\
&=& L_2(M_1L_2L_3L_5L_6L_7\cdots L_l)(p_{1,4}) =   [L_2(M_1L_2L_3L_5L_6)(p_{1,4})][L_7\cdots L_l(p_{1,4})].
\end{eqnarray*}

We can factor our evaluation matrix as $AB$ where
\[A =
\begin{bmatrix}
L_7\cdots L_l(p_{1,4}) & 0                     & \cdots & 0                     & 0                    &         \\
0                    & L_7\cdots L_l(p_{2,4})) &        & 0                     & 0                    &         \\
\vdots               &                        & \ddots &                       &                      & {\bf 0} \\
0                    &                        &        & L_7\cdots L_l(p_{4,6}) & 0                    &          \\
0                    &  0                     & \cdots & 0                     & L_7\cdots L_l(p_{1,2})&          \\
                     &                        & {\bf 0} &                       &                      & I_{2l-12} \\
\end{bmatrix},\]
where ${\bf 0}$ denotes an appropriate sized zero matrix, and $I_{2l-12}$ is the identity matrix, and
where $B$ is the matrix given by
\footnotesize{
\[B=
\begin{array}{c|ccccc|ccccccc}
       & L_2Q'_4 & L_1Q'_4 & \cdots & L_4Q'_3& L_5Q'_1& L_2Q_7 & L_1Q_7& L_2Q_8  & L_1Q_8& \cdots & L_2Q_l  & L_1Q_l\\
\hline
p_{1,4} & \star & 0     & \cdots  &  0     & \Box &  0    & 0     & 0      & 0     & \cdots  & 0     & 0     \\
p_{2,4} & 0     & \star & \cdots  & 0      & 0    &  0    & 0     & 0      & 0     & \cdots  & 0     & 0     \\
\vdots &       &       & \vdots  &        &      &       &       &        &       & \vdots  & 0     & 0     \\
p_{4,6} & 0     & 0     & \cdots  & 0      & 0    & 0     & 0     & 0      & 0     & \cdots  & 0     & 0     \\
p_{1,2} & 0     & 0     & \cdots  & 0      & \star& 0     & 0     & 0      & 0     & \cdots  & 0     & 0     \\
\hline
p_{1,7} & 0     & 0     & \cdots  & 0      & \Box & \star & 0     & 0      & 0     & \cdots  & 0     & 0     \\
p_{2,7} & 0     & 0     & \cdots  & 0      & \Box & 0     & \star & 0      & 0     & \cdots  & 0     & 0     \\
p_{1,8} & 0     & 0     & \cdots  & 0      & \Box & 0     & 0     & \star  & 0     & \cdots  & 0     & 0     \\
p_{2,8} & 0     & 0     & \cdots  & 0      & \Box & 0     & 0     & 0      & \star & \cdots  & 0     & 0     \\
\vdots &       &       & \vdots  &        &      &       &       &        &       & \vdots  & 0     & 0     \\
p_{1,l} & 0     & 0     & \cdots  & 0      & \Box & 0     & 0     & 0      & 0     & \cdots  & \star & 0     \\
p_{2,l} & 0     & 0     & \cdots  & 0      & \Box & 0     & 0     & 0      & 0     & \cdots  & 0     & \star  \\
\end{array}
\]}
\normalsize

\noindent
But now the matrix $B$ has the property that the $12 \times 12$ sub-matrix in the upper left hand
corner is exactly the same
as the matrix as in Theorem \ref{l=6}.  As a result, this sub-matrix has rank 12.  Furthermore, the lower
$(2l-12) \times (2l-12)$ sub-matrix clearly has maximal rank, and so $B$ has maximal rank.
Finally, since none of the points indexing
the first 12 rows vanish at $L_7,\ldots,L_{l}$, the matrix $A$ also has maximal rank, so our original
evaluation matrix has the desired rank of $2l$.
\end{proof}

For completeness, we now put together all the pieces to prove our
main theorem.
\begin{proof} (of Theorem \ref{maintheorem})
{ {By Lemma \ref{basicprop} and Remark \ref{dimension0}, $\mathcal{S}(d,l) =
\emptyset$ if $d < l-1$.}} The value for $\dim\mathcal{S}(4,5)$
comes from Remark \ref{lurothbound}. The main theorem of
\cite{CVT} determines when $\dim \mathcal{S}(d,l) =
\binom{d+2}{2}-1$. Theorems \ref{l=6} and \ref{l>6} cover the
remaining cases.
\end{proof}
%%%%%%%%%%%%%%%%%%%%%%%%%%%%%%%%%%%%%%%%%%%%%%%%%%%%%%%%%%%%%%%%%%%%%%

\section{Concluding remarks}

It is natural to ask whether the work of this paper can be generalized to star configurations set of points
$\XX(l)$ in $\pr^n$; see \cite{CGVT} for more on this.
Indeed, let $\Sigma_{n,d,l}$ be the incidence correspondence
\[
\Sigma_{n,d,l} = \{(\mathcal{H},\XX(l)) :
\mathcal{H} \supseteq \XX(l) \} \subseteq
\PP S_d \times \mathcal{D}_l.
\]
where now $S = k[x_0,\ldots,x_n]$ and
$\mathcal{D}_l \subseteq \check{\pr}^n \times \cdots \times
\check{\pr}^n$ ($l$ times).
Letting $\phi_{n,d,l}:\Sigma_{n,d,l} \rightarrow \PP S_d$ denote the natural
projection map, we wish to compute the dimension of the corresponding locus,
that is, $\mathcal{S}(d,l,n) = \phi_{n,d,l}(\Sigma_{n,d,l})$.

The proofs of Section \ref{section2} extend naturally to this case, thus giving us the upper
bound
\[\dim \mathcal{S}(d,l,n) \leq \min\left\{\binom{d+n}{n}-1,  \binom{d+n}{n}-\binom{l}{n}+nl-1 \right\}
\mbox{for all $d \geq l-1$}.\]  Computer experiments
suggest that this inequality is an equality for all $d \geq l-1$
with $n \geq 3$.  The results of
\cite{CGVT} already verifies part of this claim when the
minimum is $\binom{d+n}{n}-1$.  We expect
that the approach used in this paper will verify this question;
however, the difficultly is now
finding the correct evaluation matrix and determining its rank.

As an interesting aside, if this equality holds, this would imply that the L\"uroth case is
the only time $\dim \mathcal{S}(d,l,n)$ is not the expected value.
Also notice that the case of L\"uroth quartic is the only one in
the plane for which the locus of star configurations is an
hypersurface, and it is hence defined by a single equation.
Moreover, the locus of a star configurations is never zero
dimensional.

%%%%%%%%%%%%%%%%%%%%%%%%%%%%%%%%%%%%%%%%%%%%%%%%%%%%%%%%%%%%%%%%%%%%%%

\end{document}